\documentclass{amsart}
\usepackage{amssymb,amscd}

\usepackage[small,nohug,heads=vee]{diagrams} 

\title[Generic Isotriviality, Families of Group Actions, and Linearization]{Families of Group Actions, Generic Isotriviality, and Linearization}

\author{Hanspeter Kraft and Peter Russell}
\date{January 2012}

\address{Mathematisches Institut der
Universit\"at Basel,\newline
\indent Rheinsprung 21, CH-4051 Basel, Switzerland}
\email{Hanspeter.Kraft@unibas.ch}
\address{McGill University, Department of Mathematics and Statistics
\newline\indent 
805  Sherbrooke Street West, Montreal, QC, Canada, H3A 2K6}
\email{russell@math.mcgill.ca}

\thanks{The first author was partially supported by Swiss National Science Foundation.}

\newtheorem{thm}{Theorem}
\newtheorem*{thm*}{Theorem}

\newtheorem*{thmgenequiv}{Generic Equivalence Theorem}

\newtheorem*{conj*}{Conjecture}

\newtheorem*{prob*}{Problem}
\newtheorem*{satz*}{Satz}

\newtheorem{prop}{Proposition}
\newtheorem{lem}{Lemma}
\newtheorem{cor}{Corollary}
\newtheorem*{cor*}{Corollary}

\theoremstyle{definition}
\newtheorem{defn}{Definition}

\theoremstyle{remark}
\newtheorem*{rem*}{Remark}
\newtheorem{rem}{Remark}

\newcommand{\mm}{{\mathfrak m}}

\newcommand{\name}[1]{\textsc{#1\/}}
\newcommand{\NN}{{\mathbb N}}

\newcommand{\PP}{{\mathbb P}}

\newcommand{\CC}{{\mathbb C}}

\newcommand{\QQ}{{\mathbb Q}}
\newcommand{\Cst}{{{\mathbb C}^*}}

\renewcommand{\AA}{{\mathbb A}}
\newcommand{\Aone}{{\mathbb A}^{1}}
\newcommand{\Atwo}{{\mathbb A}^{2}}

\newcommand{\VVV}{\mathcal V}

\newcommand{\V}{\mathcal V}

\newcommand{\simto}{\xrightarrow{\sim}}

\newcommand{\be}{\begin{enumerate}}
\newcommand{\ee}{\end{enumerate}}

\DeclareMathOperator{\Hom}{Hom}

\DeclareMathOperator{\Aut}{Aut}

\DeclareMathOperator{\Spec}{Spec}

\DeclareMathOperator{\Char}{char}

\DeclareMathOperator{\OOO}{\mathcal O}
\newcommand{\quot}{/\!\!/}
\DeclareMathOperator{\spec}{Spec}
\DeclareMathOperator{\Mor}{Mor}
\DeclareMathOperator{\End}{End}

\DeclareMathOperator{\pr}{pr}
\newcommand{\An}{\AA^{n}}
\newcommand{\Cdot}{\dot\CC}
\DeclareMathOperator{\mor}{Mor}

\newcommand{\Ybd}{Y_{\text{\it bd}}}
\newcommand{\Ubd}{U_{\text{\it bd}}}
\newcommand{\Yt}{{\tilde Y}}
\newcommand{\Xt}{{\tilde X}}
\newcommand{\phit}{{\tilde\phi}}

\renewcommand{\phi}{\varphi}

\newcommand{\margin}[1]{}
\newcommand{\lab}[1]{\label{#1}}

\begin{document}

\begin{abstract} We prove a {\it Generic Equivalence Theorem\/} which says that two affine morphisms $p\colon S \to Y$ and $q\colon T \to Y$ of varieties  with isomorphic (closed) fibers become isomorphic under a dominant \'etale base change $\phi\colon U \to Y$. A special case is the following result. Call a morphism $\phi\colon X \to Y$ a {\it fibration with fiber $F$\/} if $\phi$ is flat and all  fibers are (reduced and) isomorphic to $F$. Then an affine fibration with fiber $F$ admits an \'etale dominant morphism $\mu\colon U \to Y$ such that the pull-back is a trivial fiber bundle: $U\times_{Y}X \simeq U\times F$. 

As an application we give short proofs of the following two (known) results: (a) {\it Every affine $\Aone$-fibration over a normal variety is locally trivial in the Zariski-topology\/} (see \cite{KaWr1985Flat-families-of-a}); (b) {\it Every affine $\Atwo$-fibration over a smooth curve is locally trivial in the Zariski-topology\/} (see \cite{KaZa2001Families-of-affine}). 

We also study families of reductive group actions on $\Atwo$ parametrized by curves and show that every faithful action of a non-finite reductive group on $\AA^{3}$ is {\it linearizable}, i.e. $G$-isomorphic to a representation of $G$.
\end{abstract}

\maketitle

\section{Introduction and main results}
\subsection{Linearization}
Our base field is the field $\CC$ of complex numbers. For a variety $X$ we denote by $\OOO(X)$ the algebra of regular functions on $X$, i.e. the global sections of the sheef $\OOO_{X}$ of regular functions on $X$. An action of an algebraic group $G$ on $X$ is called {\it linearizable\/} if $X$ is $G$-equivariantly isomorphic to a linear representation of $G$. The ``Linearization Problem'' asks if any action of a reductive algebraic group $G$ on affine $n$-space $\AA^{n}$ is linearizable. For $n=2$ the problem has a positive answer, due to the structure of the automorphism group of  $\AA^{2}$ as an amalgamated product. On the other hand there exist non-linearizable actions on certain $\AA^{n}$ for all non-commutative connected reductive groups, see \cite{Sc1989Exotic-algebraic-g}, \cite{Kn1991Nichtlinearisierba}. The open cases are commutative reductive groups, in particular tori and commutative finite groups. For a survey on this problem we refer to the literature (\cite{Kr1996Challenging-proble}, \cite{KrSc1992Reductive-group-ac}).

A very interesting case is dimension 3 where no counterexamples have occurred so far. It is known that all actions of semisimple groups are linearizable (\cite{KrPo1985Semisimple-group-a}) as well as $\Cst$-actions (see \cite{KaKoMa1997C-actions-on-C3-ar}). The following result completes the picture of reductive group actions on $\AA^{3}$.

\begin{thm}\lab{A3linearization.thm}
Every faithful action of a non-finite reductive group on $\AA^{3}$ is linearizable.
\end{thm}
We do not know if the same holds for finite group actions on $\AA^{3}$. This seems to be a very difficult problem.

\subsection{Generic isotriviality}\lab{isotriviality.subsec}
One of the basic results of our paper is the following  ``generic isotriviality'' of group actions.

\begin{thm}\lab{isotriviality.thm}
Let $\phi\colon X \to Y$ be a dominant morphism where $X,Y$ are irreducible, and let $G$ be a reductive group acting on $X/Y$. Assume that the action of $G$ on the general fiber of $\phi$ is linearizable. Then there is a dominant \'etale morphism $\mu\colon U\to Y$ such that the fiber product $X \times_{Y} U$ is $G$ isomorphic to $W \times U$ over $U$ where $W$ is a linear representation of $G$:
$$
\begin{CD}
W \times U @>\simto>>  X \times_{Y} U @>>>  X \\
@VV{\pr}V   @VVV  @VV{\phi}V \\
U  @=  U  @>{\mu}>>  Y
\end{CD}
$$
\end{thm}
As usual, the condition that ``the action of $G$ on the general fiber of $\phi$ is linearizable'' means that on an open dense subset of $Y$ all fibers $\phi^{-1}(y)$ are reduced and $G$-isomorphic to a representation of $G$. 

Theorem~\ref{isotriviality.thm} is based on a very general result, the ``Generic Equivalence Theorem'' which we formulate and prove in section~\ref{GenIso.sec}. Several special cases of this result appear in the literature, quite often in connection with so-called ``cylinder-like open sets'', but the statement seems not to be known in this general form. 

In the last paragraph we use this result to give a short and unified proof of the following results due to \name{Kambayashi-Wright} and \name{Kaliman-Zaidenberg} (see Theorem~\ref{fibration.thm}).

\begin{thm}
\be
\item If $\phi\colon X \to Y$ is a flat affine morphism with fibers $\AA^{1}$ and $Y$ normal, then $\phi$ is a fiber bundle, locally trivial in the Zariski-topology.
\item If $\phi\colon X \to Y$ is an flat affine morphism with fibers $\AA^{2}$ and  $Y$ a smooth curve, then $\phi$ is a fiber bundle, locally trivial in the Zariski-topology.
\ee
\end{thm}

\subsection{Families of group actions}
An important concept and basic tool in our paper are {\it families\/} of automorphisms and  {\it families\/} of group actions.

\begin{defn} 
Let $Z,Y$ be varieties. A {\it family of automorphisms of $Z$ parametrized by $Y$} is an automorphism $\Phi$ of $Z \times Y$ such that the  the projection $\pr\colon Z\times Y \to Y$ is invariant. We use the notation $\Phi = (\Phi_{y})_{y\in Y}$ where $\Phi_{y}$ is the induced automorphism of the fiber $Z\times\{y\}$ which we identify with $Z$. 

Similarly, for an algebraic group $G$, a {\it family of $G$-actions on $Z$ parametrized by $Y$\/} is a $G$-action $\Phi$ on $Z\times Y$ such that the projection $\pr\colon Z\times Y \to Y$ is $G$-invariant. Again we use the notation $\Phi = (\Phi_{y})_{y\in Y}$ where $\Phi_{y}$ is the $G$-action on the fiber $Z\times\{y\}$ identified with $Z$.
\end{defn}
\par\medskip
Using an equivariant form of \name{Sathaye}'s famous Theorem (see Lemma~\ref{linearization.lem}) we obtain the following result about linearization of families of two dimensional representations.

\begin{thm}\lab{linearization.thm}
Let $G$ be a reductive group, and let $\Phi$ be a family of $G$-actions  on $\AA^{2}$ parametrized by a 
factorial affine curve $C$. Then the family is simultaneously linearizable, i.e., $\AA^{2}\times C$ is $G$-isomorphic to $V\times C$ where $V$ is a two-dimensional linear representation of $G$.
\end{thm}
This has the following consequence. Recall that a {\it variable\/} of $\AA^{n}$ is a regular function $f$ on $\AA^{n}$ which appears in an algebraic independent system of generators of the polynomial ring $\OOO(\AA^{n})$.

\begin{cor}\lab{variable.cor}
A reductive groups action on $\AA^{3}$ fixing a variable is linearizable.
\end{cor}

We conjecture that this holds in the more general situation where the reductive group action on $\AA^{3}$ {\it normalizes} a variable.

\subsection{Ind-varieties and ind-groups}
In order to explain the next application, let us recall that the group $\Aut(\AA^{n})$ of polynomial automorphisms of affine $n$-space has the structure of an {\it ind-group} (see \cite{FuMa2010A-characterization} or \cite{Ku2002Kac-Moody-groups-t}; this notion goes back to Shafarevich who called this objects {\it infinite dimensional varieties\/} or {\it groups}, see \cite{Sh1966On-some-infinite-d,Sh1981On-some-infinite-d,Sh1995Letter-to-the-edit}). 

\begin{defn} 
An {\it ind-variety\/} $\V$ is a set together with subsets $\V_{1}\subset \V_{2}\subset \V_{3}\subset \cdots$ such that the following holds:
\be
\item $\V = \bigcup_{k}\V_{k}$;
\item Each $\V_{k}$ has the structure of a variety;
\item $\V_{k}\subset \V_{k+1}$ is closed in the Zariski-topology for all $k$.
\ee
An ind-variety $\V$ has a natural topology where $S\subset \V$ is open (resp. closed or locally closed) if and only if $S_{n}:=S\cap \V_{k} \subset \V_{k}$ is open (resp. closed or locally closed). Obviously, a locally closed subset $S \subset \V$ has a natural structure of an ind-variety. An ind-variety $\V$ is called {\it affine\/} if all $\V_{k}$ are affine. It is also clear how to define {\it morphisms\/} and {\it isomorphisms\/} of ind-varieties as well as {\it ind-groups}.
\end{defn}
Basic objects are $\CC$-vector spaces $V$ of countable dimension which can be given the structure of an (affine) ind-variety by choosing an increasing sequence of finite dimensional subspaces $V_{k}$ such that $V = \bigcup_{k} V_{k}$. The structure is independent of the choice of this sequence in the sense that for any two such choices the identity map is an isomorphism. For example, if $X$ is an affine variety and $W$ a finite dimensional vector space, then $\Mor(X,W) = \OOO(X)\otimes W$ is an ind-variety. Choosing a closed embedding $X\subset W$ one easily sees that $\End(X) = \Mor(X,X)\subset \Mor(X,W)$ is closed, so that 
$\End(X)$ is an (affine) ind-variety where the structure does not depend on the embedding $X \subset W$.

An important case is $\End(\AA^{n})=\CC[x_{1},\cdots,x_{n}]^{n}$ where the ind-structure is usually given by $\End(\AA^{n}]_{k}:= (\CC[x_{1},\cdots,x_{n}]_{\leq n})^{k}$, the endomorphisms of degree $\leq k$ where the degree of  $\phi=(f_{1},\ldots,f_{n})$ is defined to be $\deg\phi:= \max(\deg f_{i})$. One can show that
$\Aut(\AA^{n}) \subset  \End(\AA^{n})$ is locally closed, i.e. the automorphisms $\Aut(\AA^{n})_{k}$ of degree $\leq k$ are locally closed in $\End(\AA^{n})_{k}$. Moreover, multiplication and inverse are morphisms of ind-varieties so that $\Aut(\AA^{n})$ is indeed an ind-group. (For the inverse one has to use the formula $\deg\phi^{-1}\leq(\deg\phi)^{n-1}$ due to \name{Offer Gabber}, see \cite{BaCoWr1982The-Jacobian-conje}.) 

Using this structure it is easy to see that a family $\Phi=(\Phi_{y})_{y\in Y}$ of automorphisms of $\AA^{n}$ parametrized by $Y$ defines a morphism $\tilde\Phi\colon Y\to \Aut(\AA^{n})$, $y\mapsto\Phi_{y}$, and vice versa. Similarly, a family of group  actions of a reductive group $G$ parametrized by $Y$ is the same as a morphism $Y \to \mor(G,\Aut(\AA^{n}))$ such that the image belongs to $\Hom(G,\Aut(\AA^{n}))$ where $\mor(G,\Aut(\AA^{n})) = \bigcup_{k}\mor(G,\Aut(\AA^{n})_{k})$ also has a natural structure of an ind-variety.

\bigskip
\section{Generic equivalence and generic isotriviality}\lab{GenIso.sec}

Our first result concerns the generic equivalence of two morphisms having the same fibers. This holds under very general conditions. The main ingredient is the following lemma which should be well-known. Let $p\colon X \to Y$ be a dominant morphism between affine $k$-varieties where $k$ is algebraically closed and  $Y$ irreducible. Then there is a field $k_{0}\subset k$ which is finitely generated over the prime field and a morphism $p_{0}\colon X_{0}\to Y_{0}$ of affine $k_{0}$-varieties with a cartesian diagram
$$
\begin{CD}
X @>p>> Y @>>> \Spec k\\
@VVV @VVV  @VVV\\
X_{0} @>p_{0}>> Y_{0} @>>> \Spec k_{0}\\
\end{CD} 
$$

\begin{lem}\lab{basic.lem}
In the notation above denote by $\omega\colon\Spec K_{0}\to Y_{0}$ the generic point of $Y_{0}$ and by $(X_{0})_{\omega}:=p_{0}^{-1}(\omega)$ the generic fiber of $p_{0}$. Then every $k_{0}$-embedding $K_{0}\hookrightarrow k$ defines a closed point $y\in Y$ and an isomorphism
$$
(X_{0})_{\omega} \times_{\Spec K_{0}}\Spec k \simto X_{y}:=p^{-1}(y).
$$
\end{lem}
\begin{proof} The $k_{0}$-embedding $\OOO(Y_{0}) \hookrightarrow K_{0}\hookrightarrow k$ defines a $k$-homomorphism $\OOO(Y)=\OOO(Y_{0})\otimes_{k_{0}}k \to k$, hence a closed point $\iota_{y}\colon \{y\}\to Y$, and we obtain the following commutative diagram where $X_{y}=p^{-1} (y)$ is the (schematic) fiber of $y$.

\begin{diagram}
X_{y} &\rTo(8,0) &&&&&(X_{0})_{\omega}\\
\dTo(0,8)&\rdTo &&&& \ldTo&\dTo(0,8)\\
&& X &\rTo& X_{0}\\
&&\dTo_{p}&&\dTo_{p_{0}}\\
&&Y&\rTo&Y_{0}\\
&\ruTo_{\iota_{y}} &\dTo(0,4)&&\dTo(0,4)& \luTo_{\omega}\\
\Spec k & \rTo(8,0) &&&&& \Spec K_{0}\\
&\rdTo_{=} &&&& \ldTo\\
&&\Spec k &\rTo & \Spec k_{0}\\
\end{diagram}
It follows that the outer diagram is cartesian:
$$
(X_{0})_{\omega}\times_{\Spec K_{0}}\Spec k \simeq X_{0}\times_{Y_{0}}\Spec k \simeq X \times_{Y}\Spec k = X_{y}.
$$
\end{proof}

\begin{thmgenequiv}\label{genequiv.thm} 
Let $k$ be an algebraically closed field of infinite transcendence degree 
over the prime field. Let $p\colon S\to Y$ and $q\colon T\to Y$ be two affine morphisms where $S,T$ and $Y$ are $k$-varieties. Assume that for all $y\in Y$ the two (schematic) fibers $S_y:=p^{-1}(y)$ and $T_y:=q^{-1}(y)$ are isomorphic. Then there is a dominant \'etale morphism $\phi\colon U\to Y$ and an
isomorphism $S\times_Y U \simeq T\times_Y U$ over $U$:
\begin{diagram}
S & \lTo &  S\times_Y U && \rTo^{\simeq}&& T\times_Y U &\rTo & T\\
\dTo &&& \rdTo && \ldTo &&& \dTo \\
Y && \lTo^{\phi} && U && \rTo^{\phi} && Y
\end{diagram}
\end{thmgenequiv}
\begin{rem}\lab{Gequiv.rem}
Under the assumptions of the proposition assume in addition
that an algebraic group $G$ acts on $S$ and $T$ such that $p$ and $q$ are both
invariant  and that the isomorphisms
$\phi_y\colon S_y \simto T_y$ can be chosen to be  $G$-equivariant. Then the
proposition holds  $G$-equivariantly, i.e., there is an \'etale 
morphism $U \to Y$ and a $G$-equivariant isomorphism $S\times_Y U \simeq
T\times_Y U$.
\end{rem}
\begin{rem}
We do not know if the Theorem holds for all algebraically closed fields, e.g. for $\bar\QQ$.
\end{rem}
\begin{proof} We can assume that $Y$ is affine and irreducible. 
Clearly, the whole setting is defined over a field $k_0$ which is
finitely generated over the prime field. This means that there are $k_0$-varieties
$Y_0,S_0,T_0$ and morphisms $p_0\colon S_0 \to Y_0$, $q_0\colon T_0
\to Y_0$ which become $p\colon S \to Y$, $q\colon T \to Y$ under the base
change $k/k_0$, i.e., the following diagrams are cartesian:
\begin{equation*}  
\begin{CD} 
Y_0@>>>\spec k_0@.\qquad\qquad S_0 @>p_0>> Y_0 @<q_0<<T_0\\ 
@AAA @AAA\qquad\qquad@AAA@AAA@AAA\\ 
Y @>>> \spec k @.\qquad\qquad S @>p>> Y @<q<<T
\end{CD}  
\end{equation*}
Let $\omega\colon \Spec K_{0} \to Y_0$ be the generic point of $Y_{0}$. By assumption on the field $k$ we can
embed $K_0$ into $k$ (over $k_0$). According to Lemma~\ref{basic.lem} we get a closed point $\iota\colon\{y\} \to Y$ and isomorphisms
\begin{equation}
{S_0}_\omega\times_{\spec K_0} \spec k \simeq S_y  \simeq T_{y}\simeq {T_0}_\omega\times_{\spec K_0} \spec k.
\end{equation}
This implies that there is a finite field extension $L_0/K_0$ and an isomorphism
$$
{S_0}_\omega \times_{\spec K_0}\spec L_0 \simeq
{T_0}_\omega \times_{\spec K_0}\spec L_0.
$$
In fact, in $(1)$ we can first replace $k$ by a finitely generated $K_{0}$-algebra $A$ and then pass to 
$L_{0}:= A/\mm$ where $\mm\subset A$ is a maximal ideal.

It follows that there is a finite field extension $L$ of $K=k(Y)$, the field of rational functions 
on $Y$, and an isomorphism
$$
{S}_\omega \times_{\spec K}\spec L \simeq
{T}_\omega \times_{\spec K}\spec L.
$$
where again $S_\omega$ and $T_\omega$ denote the generic fibers of $p$
and $q$ (over $\spec K$). Since ${S}_\omega \times_{\spec K}\spec L =
S\times_Y\spec L$ there is a variety $X$  and
a dominant morphism $X \to Y$ of finite degree $[L:K]$ such that $S\times_Y X
\simeq T\times_Y X$.
\end{proof}

Using the equivariant form of this result (see Remark~\ref{Gequiv.rem} above) we can now prove Theorem~\ref{isotriviality.thm}.
\begin{proof}[Proof of Theorem~\ref{isotriviality.thm}]
The assumptions of the theorem imply that there is an open dense set $U\subset X$ with the following properties:
\be
\item $U$ is smooth;
\item The fibers $\phi^{-1}(u)$ for $u\in U$ are reduced and isomorphic to $\CC^{n}$ where $n:=\dim Y - \dim X$;
\item The action of $G$ on a fiber $\phi^{-1}(u)$ for $u\in U$ is linearizable.
\ee
To finish the proof using the Equivariant Generic Equivalence Theorem (Remark~\ref{Gequiv.rem}) we have to show the following:
\be
\item[(d)] For all $u\in U$ the fiber $\phi^{-1}(u)$ is $G$-isomorphic to a fixed representation $W$ of $G$.
\ee
In fact, $\phi\colon\phi^{-1}(U)\to U$ is smooth and surjective and  the tangent space $T_{x_{0}}Y$ in a fixed point $x_{0}\in\phi^{-1}(U)^{G}$ has a $G$-stable decomposition $T_{x_{0}}Y = T_{x_{0}}F_{0}\oplus V$ where $F_{0}:= \phi^{-1}(\phi(x_{0}))$, and $d\phi_{x_{0}}\colon V \simto T_{\phi(x_{0})}X$, since $G$ is reductive. Moreover, there is a $G$-equivariant morphism $\mu\colon T_{x}F_{0}\oplus V \to Y$ sending $(0,0)$ to $x_{0}$ which is \'etale in a neighbourhood of $(0,0)$. This implies that for all fixed points $x$ in a neighbourhood of $x_{0}$ the tangent representation $T_{x}F$, $F := \phi^{-1}(\phi(x))$, is isomorphic to $T_{x_{0}}F_{0}$. Thus all fibers in a neighbourhood of $\phi(x_{0})$ are $G$-isomorphic to the same representation.
\end{proof}

\bigskip
\section{Families of group actions on $\AA^{2}$}\lab{families.sec}

We start with a crucial result on families of group actions on $\AA^{2}$ where we use in essential way the amalgamated product structure of $\Aut(\AA^{2})$. We do not know how to generalize this to higher dimension.

\begin{lem}\lab{opentrivial.lem}
Let $G$ be a reductive group and let  $\Phi$ be a family of $G$-actions on $\AA^{2}$ parametrized by $Y$ where $Y$ is an irreducible affine variety. Then there is an open dense set $U \subset Y$ such that the family $\Phi|_{U}$ is equivalent to the constant family of a (2-dimensional) linear representation $V$ of $G$, i.e., there is a $G$-equivariant isomorphism of $\AA^{2}\times U$ with $V\times U$.
\end{lem}

\begin{proof} 
It is known that $\Aut(K[x,y])$ has the structure of an amalgamated product for any field $K$ of characteristic zero (\cite{Ku1953On-polynomial-ring}). This implies that every reductive $K$-group action on $\AA^{2}_{K}$ is linearizable (\cite{Ka1979Automorphism-group}). Taking for $K$ the field of fractions of $\OOO(Y)$ this means that there exist  $x_1,y_1 \in K[x,y]$ such that $K[x,y] = K[x_1,y_1]$ and that $K x_1 \oplus K y_1$ is stable under $G_{K}$. Since every representation of $G_{K}$ is defined over $\CC \subset K$ (see e.g.  \cite[Corollary~II.2.9]{Ja2003Representations-of}) we can assume that $\CC x_{1} \oplus \CC y_{1}$ is $G$-stable.
Clearly, there is a $t\in\OOO(Y)$ such that $\OOO(Y)_{t}[x,y] = \OOO(Y)_{t}[x_1,y_1]$ and so $\AA^{2}\times U$ is isomorphic to $V \times U$ as a $G$-varieties where $U := Y_{t}$ and $V := (\CC x_1 \oplus \CC y_1)^{*}$, the dual representation. 
\end{proof}

\begin{proof}[Proof of Theorem~\ref{linearization.thm}]
Set $A:=\OOO(C)$ so that $\OOO(\AA^{2}\times C) = A[x,y]=:R$. We have seen in Lemma~\ref{opentrivial.lem} that there exist  $t \in A$ and $x_1,y_1 \in A_{t} := A[t^{-1}]$ such that $A_{t}[x,y] =A_{t}[x_1,y_1]$ and that $\CC x_1 \oplus \CC y_1$ is $G$-stable. Clearly, $C$ is obtained from $C_{t}$ by adding a finite number of points: $C = C_{t}\cup\{c_{1},c_{2},\ldots,c_{k}\}$. Moreover, every open set $C_{j}:=C_{t}\cup\{c_{1},c_{2},\ldots,c_{j}\}$ is an affine factorial curve. Hence, by induction, we can assume that $C$ is obtained from $C_{t}$ by adding a single point $c_{0}$, i.e., that $t$ has a simple zero in $c_{0}\in C$ and that $At \subset A$ is the maximal ideal at $c_{0}$. Now the claim follows from the next lemma. (We only need the special case where the $G$-action on $C$ trivial.)
\end{proof}

\begin{lem}\lab{linearization.lem} 
Let $C$ be an affine smooth curve and let $G$ be a reductive group acting on $X:=\AA^{2} \times C$ 
such that the projection $p\colon X \to C$ is $G$-equivariant. Let $c_{0}\in C$ be a fixed point of $G$  and set $C':=C\setminus\{c_{0}\}$. Assume that the following holds:
\be
\item There is a generator $t$ of the maximal ideal $\mm_{c_{0}}\subset \OOO(C)$ such that $\CC t$ is $G$-stable;
\item $X':=p^{-1}(C')$ is $G$-isomorphic to the product $W \times C'$ where $W$ is a two-dimensional representation of $G$.
\ee 
Then $X$ is $G$-isomorphic to $W \times C$.
\end{lem}

\begin{proof}
Set $A:=\OOO(C)$ so that $\OOO(\AA^{2}\times C) = A[x,y]=:R$. By assumption there exist $x_1,y_1 \in R$  such that $R_{t}=A_{t}[x,y] =A_{t}[x_1,y_1]$ and that $\CC x_1 \oplus \CC y_1$ is $G$-stable. Denoting by 
$\bar{x_1}, \bar{y_1}$ the residue classes in $\bar R := R/Rt = \CC[x,y]$ we obtain a linear
$G$-homomorphism  $\rho\colon\CC x_1 \oplus \CC y_1 \to \CC \bar{x_1} + \CC 
\bar{y_1} \subset \bar R$. Deviding $x_{1}$ and $y_{1}$ by the same power of $t$ we can assume that $\rho$ is non-zero. If the image $\CC\bar{x_1} + \CC\bar{y_1}$ has dimension
$1$ then the kernel of $\rho$ is a one-dimensional representation $\CC h$ of $G$, so that  
$\CC x_{1}\oplus\CC y_{1}=\CC h \oplus \CC h'$ where $\CC h'$ is $G$-stable. Now we can divide $h$ by a suitable power of $t$.  In this way we
arrive at a situation where $\dim(\CC\bar{x_1} + \CC\bar{y_1}) = 2$.  Define 
\begin{align}
n(x_1,y_1) &:= \min\{n\geq 0\mid t^nx \in A[x_1,y_1]\},\\
m(x_1,y_1) &:= \min\{m\geq 0\mid t^my \in A[x_1,y_1]\},\notag
\end{align}
and set 
\begin{equation} 
N(x_1,y_1):= n(x_1,y_1)+m(x_1,y_1).
\end{equation}
Then we find the following expressions
\begin{equation}\tag*{($*$)}\label{eq1}
t^{n(x_1,y_1)} x = \sum_{i,j} a_{ij}{x_1}^i{y_1}^j, \quad t^{m(x_1,y_1)}y = \sum_{i,j} b_{ij}{x_1}^i{y_1}^j
\end{equation}
where $a_{ij}, b_{ij} \in A$, and not all $a_{ij}$ and not all $b_{ij}$ belong to the
maximal ideal $At$.  If $N(x_1,y_1)=0$ we are done. Otherwise it follows from \ref{eq1} that
$\bar{x_{1}}, \bar{y_{1}}$ are algebraically dependent.
 
Denote by $F \in \CC[w,z]$ the minimal equation
$F(\bar{x_1}, \bar{y_1})=0$. Clearly, $F$ is the generator of the kernel of the
canonical homomorphism $\phi\colon \CC[w,z] \to \bar R$ given by
$\phi(w):=\bar{x_1}$ and $\phi(z):=\bar{y_1}$. Now it follows from \cite[Remark~2.1]{Sa1983Polynomial-ring-in}  
that $\CC[\bar{x_1}, \bar{y_1}]\subset \bar R$ is a polynomial
ring in one variable, or, equivalently, that $F$ is a variable in
$\CC[w,z]$, i.e., there is a $H\in\CC[w,z]$ such that $\CC[w,z]=\CC[F,H]$.

If we define a (linear) $G$-action on $\CC[w,z]$
by using the same matrices as for the representation on $\CC x_1\oplus \CC y_1$, then the
homomorphism $\phi$ is obviously $G$-equivariant, hence the
kernel is $G$-stable. This implies that $\CC F \subset \CC[w,z]$ is 
$G$-stable. Now we use the fact that there is a uniquely defined second
variable $H \in\CC[w,z]$ (up to an additive constant) which has
lower degree than $F$ (see \cite[Theorem 3(1)]{Sa1983Polynomial-ring-in}). It follows that $\CC H + \CC\subset \CC[w,z]$ is
$G$-stable and so $\CC(H + \alpha)$ is $G$-stable for a suitable
$\alpha\in\CC$. 

Putting 
$x_2:=F(x_1,y_1)\in R$ and $y_2:=H(x_1,y_1)+\alpha\in R$  we see that $\CC x_2$
and $\CC y_2$ are $G$-stable lines in $A$. Moreover, we have
$\CC[x_2,y_2] = \CC[x_1,y_1] \subset R$ and so  $A[x_2,y_2] = A[x_1,y_1]$. Since
$\bar {x_2} = F(\bar x_{1},\bar y_{1})=0$ we can divide $x_2$ by a suitable power of $t$ such that
$x_3:=\frac{x_2}{t^s} \in R \setminus Rt$ for some $s>0$. Similarly, $y_{3}:= \frac{y_{2}}{t^{r}}\in R \setminus Rt$ 
for some $r\geq 0$.

In order to see that this procedure will finally stop we calculate the number
$N(x_3,y_3)$. Since $A[x_{2},y_{2}] = A[x_1,y_1]$ we have $n(x_{2},y_{2}) =
n(x_1,y_1)$ and $m(x_{2},y_{2})=m(x_1,y_1)$, and one of them is $>0$, say $n(x_1,y_1)>0$. 
Using the first equation in \ref{eq1} for
$x_2= t^s x_3$ and $y_{2}=t^{r}y_{3}$  we see that
$\overline{\sum_j a_{0j}y_{2}^j}=0$. It follows that either $r>0$ or $\overline{a_{0j}}=0$  for all
$j$. In both cases we can divide both sides of the equation by
$t$ and so $n(x_3,y_3) < n(x_{2},y_{2})$, hence $N(x_{3},y_{3}) < N(x_{1},y_{1})$.
\end{proof}

\begin{rem} The crucial step in the proof above  is \name{Sathaye}'s result  showing that
$\CC[\bar{x_1}, \bar{y_1}]\subset \CC[x,y]$ is a polynomial
ring in one variable in case $\bar{x_1}, \bar{y_1}$ are algebraically
dependent. It is interesting to remark that this result is not needed in
case $G$ is non-commutative, since there is no faithful action of a
non-commutative group $G$ on $\CC[\bar x_{1},\bar y_{1}]$ in case this algebra is of dimension 1, because there is no
faithful action of $G$ on a (rational) curve.
\end{rem}

\bigskip
\section{Linearization of group actions on $\AA^{3}$}
We now give the proof of Theorem~\ref{A3linearization.thm} stating that every faithful action of a non-finite reductive 
group $G$ on $\AA^{3}$ is linearizable.
\begin{proof} 
(a)  It follows from \name{Luna}'s  Slice Theorem that  any action of a
reductive group on $\AA^n$ with zero-dimensional quotient $\AA^{n}\quot G$ is linearizable.
Moreover, if $\dim \AA^3\quot G = 1$, then the result is explicitely stated in \cite[Chap.~VI,
\S3, Theorem 3.2(5)]{KrSc1992Reductive-group-ac}. 
Finally, if $G^{0}$ is not $\Cst$, then the quotient $\AA^3\quot G$ has dimension  $\leq 1$ and we are done in this case.

(b) If $G \simeq \Cst$, then this is the main result of \cite{KaKoMa1997C-actions-on-C3-ar}. So we are left with the case of a non-connected $G$ 
such that $G^{0}\simeq \Cst$. 

(c) We fix an identification $G^{0}= \Cst$. By (b) we can assume that the action of $\Cst$ is linear with weights $n_{1}\geq n_{2}\geq 0 > n_{3}$, i.e. $t(x,y,z) = (t^{n_{1}}\cdot x, t^{n_{2}}\cdot y,t^{n_{3}}\cdot z)$, since in all other cases the quotient $\AA^{3}\quot G$ has dimension $\leq 1$, and so we are done by (a).

(d) Let us first consider the case where $n_{2}>0$. Then the hyperplane $U$ given by $z=0$  has the following description: 
$$
U = \{v\in \AA^{3} \mid \lim_{t\to 0}tv = 0\}.
$$
This implies that every $g\in G$ commutes with $\Cst$ and therefore stabilizes $U$. In fact, if $g$ does not commute with $\Cst$ then $g t g^{-1} = t^{-1}$ for all $t\in\Cst$ and so 
$$
gU = \{v\in \AA^{3} \mid \lim_{t\to \infty}tv = 0\}.
$$
This is a contradiction since  the right hand side equals the line $\{x=y=0\}$. It follows that $\CC z \subset \OOO(\AA^{3})$ is $G$-stable: $gz = \chi(g)\cdot z$ where $\chi$ is a character of $G$.  Thus the projection $p\colon \AA^{3}\to \CC_{\chi}$, $(x,y,z)\mapsto z$, is $G$-equivariant. Define $H:=\ker\chi$ and set $\Cdot := \CC_{\chi}\setminus\{0\}$. Then $p^{-1}(\Cdot)$ is $G$-isomorphic to the associated bundle $B:=G *_{H} p^{-1}(1)$. The action of $H$ on $p^{-1}(1) \simeq\AA^{2}$ is linearizable and so $B \simeq W\times \Cdot$ as an $H$-variety where $W$ is a two-dimensional representation of $H$ and $H$ acts trivially on $\Cdot$. Thus, by Lemma~\ref{linearization.lem}, the action of $H$ on $\AA^{3}$ is linearizable: $\AA^{3}$ is $H$ isomorphic to $W \times \CC$. In particular, the hyperplane $U$ is $H$-isomorphic to $W$ which implies that the representation of $H$ on $W$ can be extended to a representation of $G$. As a consequence, the associated bundle $B$ splits into a product:
$$
p^{-1}(\Cdot) \simto W \times \Cdot_{\chi} \text{ as a $G$-variety}.
$$
Now we can again apply Lemma~\ref{linearization.lem} and the claim follows.

(e) We are left with the case $m_{2}=0$. Here we have the following two hyperplanes
\begin{align*}
U_{0}&:=\{z=0\} =  \{v\in \AA^{3} \mid \lim_{t\to 0}tv \text{ exists}\},\\
U_{\infty}&:=\{x=0\} =  \{v\in \AA^{3} \mid \lim_{t\to \infty}tv \text{ exists}\}.
\end{align*}
Clearly, $U_{0}\cup U_{\infty}$ is stable under $G$ and therefore $\CC x \oplus \CC z\subset \OOO(\AA^{3})$ is a $G$-stable subspace. This implies that the linear projection $p\colon \AA^{3}\to \CC^{2}$, $(x,y,z)\mapsto (x,z)$ is $G$-equivariant. Now \cite[Proposition 1]{KrKu1996Equivariant-affine} implies that the action of $G$ on $\AA^{3}$ is linearisable.
\end{proof}

\bigskip
\section{Fibrations and fiber bundles}
We start with the following definitions.
\begin{defn} Let $X,Y,F$ be varieties. A morphism $\phi\colon X \to Y$ is called  {\it fibration with fiber $F$\/} if  $\phi$ is flat and all fibers of $\phi$ are reduced and isomorphic to $F$. If, in addition, $\phi$ is an affine morphism, hence $F$ is affine, then we say that  $\phi$ is an {\it affine fibration with fiber $F$\/}.

A morphism $\phi\colon X \to Y$ is called a {\it fiber bundle with fiber $F$\/} if $\phi$ is locally trivial in the \'etale topology with fiber $F$, i.e. for every $y\in Y$ there is an \'etale morphism $\mu\colon U \to Y$ such that $y\in\mu(U)$ and the $U\times_{Y}X \simto U\times F$ over $U$.
\end{defn}
The following problem goes back to a paper of \name{Dolgachev-Weisfeiler} \cite{VeDo1974Unipotent-group-sc}.
\begin{prob*} 
Is it true that every (affine) fibration with fiber $\An$ is a fiber bundle?
\end{prob*}
After several attempts the case of $\AA^{1}$-fibrations was solved in \cite{KaWr1985Flat-families-of-a}. For $\AA^{2}$-bundles there is a positive answer in case the base $Y$ is a smooth curve, see \cite{KaZa2001Families-of-affine}. We will give a short unified proof for both results, partially based on our Generic Isotriviality Theorem in section~\ref{GenIso.sec}.

\begin{thm}\lab{fibration.thm}
\be
\item Let $\phi\colon X \to Y$ be an affine fibration with fiber $\AA^{1}$. If $Y$ is normal, then $\phi$ is a fiber bundle, locally trivial in the Zariski-topology.
\item If $\phi\colon X \to Y$ is an affine fibration with fiber $\AA^{2}$ and  $Y$ a smooth curve, then $\phi$ is a fiber bundle, locally trivial in the Zariski-topology.
\ee
\end{thm}

\begin{rem} The normality assumption in part (a) and (b) is essential. \name{Nori} gave an example of an $\AA^{1}$-bundle over the cusp $C:=\VVV(y^{2}-x^{3})\subset \CC^{2}$ which is not a fibration (see \cite[section 3.4]{KaWr1985Flat-families-of-a}). Consider the normalisation $\eta\colon\AA^{1}\to C$ given by $t\mapsto(t^{2},t^{3})$, and define $\phi\colon\AA^{1}\to C\times \PP^{1}$ by $t\mapsto (\mu(t),t)$. This is a closed embedding and $X:=C\times\PP^{1}\setminus\phi(\AA^{1})$ is an affine variety. If follows that the projection $p\colon X \to C$ is an $\AA^{1}$-fibration, but there is no neighborhood $U$ of the singular point of $C$ such that $p^{-1}(U)\to U$ is a trivial bundle.
\end{rem}

\begin{rem} The main result of \name{Kambayashi-Wright} in \cite{KaWr1985Flat-families-of-a} is a variant of our Theorem~\ref{fibration.thm}(a). In their setting $Y$ is a Noetherian scheme, $\phi$ is faithfully flat of finite type and the fiber of every $y\in Y$ is isomorphic to $\Aone_{\kappa(y)}$. It is not difficult to see, using the generic isotriviality, that this implies our result.
\end{rem}

\begin{rem} The first two unknown cases are $\AA^{3}$-fibrations over smooth curves and $\Atwo$-fibrations over smooth surfaces. In his thesis \name{V\'en\'ereau} constructed a polynomial $p(x,y,z,w)$ with the property that $p\colon \CC^{4}\to \CC$ is an $\AA^{3}$-fibration and $(p,w)\colon\CC^{4}\to \CC^{2}$ is an $\AA^{2}$-fibration, but in both cases it is unknown if the fibration is locally trivial in a neighbourhood of $0$, cf. \cite{KaZa2004Venereau-polynomia}.
\end{rem}

\begin{rem}\lab{BCW.rem} At this point we should mention the following very interesting result due to \name{Bass, Conell} and \name{Wright} \cite{BaCoWr1976Locally-polynomial2}: {\it Every $\An$-bundle over an affine variety which is locally trivial in the Zariski topology has the structure of a vector bundle.} As a consequence we get the following corollary.
\end{rem}

\begin{cor} 
\be
\item Let $\phi\colon X \to Y$ be an affine fibration with fiber $\AA^{1}$. If $Y$ is affine and normal, then $\phi$ has the structure of a line bundle.
\item If $\phi\colon X \to Y$ is an affine fibration with fiber $\AA^{2}$ and  $Y$ an affine smooth curve, then $\phi$ has the structure of a vector bundle of rank 2.
\ee
\end{cor}

It is clear from the definition that a fibration $\phi\colon X \to Y$ with a smooth fiber $F$ is a smooth morphism (see \cite[III.10 Definition]{Ha1977Algebraic-geometry}). In particular, $X$ is normal in case $Y$ is normal. In fact, we have an isomorphism  of the completions $\widehat\OOO_{x}\simeq\widehat\OOO_{y}[\![t_{1},\ldots,t_{n}]\!]$ where $y=\phi(x)$ and $n=\dim F$.

We will also use the following well-known fact. If, for a given point $y\in Y$, there is a smooth morphism $\psi\colon Z\to Y$ such that $y\in\psi(Z)$ and $Z \times_{Y}X \simeq Z\times F$ over $Z$, then there is also an \'etale morphism $\eta\colon U \to Y$ with the same property. 

Finally, every fiber bundle with fiber $\Aone$ is locally trivial in Zariski-topology, because the automorphism group of $\Aone$ is a special group (see \cite{KrSc1992Reductive-group-ac}).

The two basic results which we will need in the proof are the following. If $S$ is a ring and $n\in\NN$, we use $S^{[n]}$ to denote the polynomial ring over $S$ in $n$ variables.

\begin{prop}\lab{noforms.prop}
Let $L/K$ be a field extension where $\Char K=0$. Let $R$ be a finitely generated $K$-algebra such that $L\otimes_{K}R \simeq L^{[n]}$. If  $n=1$ or $n=2$, then $R\simeq K^{[n]}$.
\end{prop}

\begin{cor} 
Every fiber bundle with fiber $\AA^{1}$ or $\AA^{2}$ is locally trivial in the Zariski topology.
\end{cor}

\begin{prop}\lab{discval.prop}
Let $A$ be a discrete valuation ring with quotient field $Q(A)=K$, maximal ideal $\mm$ and residue field $k = A/\mm$ where $\Char k = 0$. Let $R\supset A$ be a domain,  finitely generated and  flat over $A$ such that  $K\otimes_{A}R\simeq K^{[n]}$ and  $k\otimes_{A}R = R/\mm R \simeq k^{[n]}$. If $n=1$ or $n=2$, then $R \simeq A^{[n]}$.
\end{prop}
For both propositions the case $n=1$ is well-known and not difficult to prove. As for the case $n=2$ the first proposition follows from the amalgamed product structure of the automorphism group of the algebra $K[x,y]$ (see \cite{Ka1975On-the-absence-of-}), and the second proposition is proved in \cite[Theorem~1]{Sa1983Polynomial-ring-in}.

\begin{rem}\lab{discrete.rem}
In case $n=1$ there is the following stronger version of Proposition~\ref{discval.prop} which does not assume that the morphism is affine, see \cite[Proposition~1.4]{KaWr1985Flat-families-of-a}. {\it If $\phi\colon X \to \Spec A$ is faithfully flat of finite type such that the generic fiber and the special fiber are both affine lines, then $X\simto \Spec A[t]$.}
\end{rem}

The proof of Theorem~\ref{fibration.thm} will be given in a series of lemmas. Let $\phi\colon X \to Y$ be an affine fibration with fiber $\AA^{n}$ where $n=1$ or $=2$. We can clearly assume that $Y$ is affine.
\begin{lem}\lab{lemA}
There is a dense open set $U\subset Y$ such that $\phi^{-1}(U)\to U$ is a trivial fiber bundle.
\end{lem}
\begin{proof} By the Generic Isotriviality Theorem in section~\ref{GenIso.sec} there is an \'etale morphism $U\to Y$ where $U$ is affine such that the bundle $U\times_{Y}X \to U$ is trivial. Therefore, 
$\CC(U)\otimes_{\CC(Y)}\OOO(X)\simeq \CC(U)^{[n]}$,  and so $\CC(Y)\otimes_{\OOO(Y)} \OOO(X)\simeq \CC(Y)^{[n]}$ by Proposition~\ref{noforms.prop}. Hence there is an $f\in\OOO(Y)$ such that $\OOO(X)_{f}\simeq \OOO(Y)_{f}^{[n]}$.
\end{proof}
\begin{lem}\lab{lemB}
Now assume that $Y$ is normal.
Let $D\subset Y$ be an irreducible hypersurface such that $\OOO(Y)_{D}$ is normal. Then there is an $f\in\OOO(Y)\setminus I(D)$ such that  $\phi\colon X_{f}\to Y_{f}$ is a trivial bundle.
\end{lem}
\begin{proof}
The morphism $E:=\phi^{-1}(D) \to D$ is a fibration with fiber $\AA^{n}$, and so $\CC(D)\otimes_{\OOO(D)}\OOO(E) \simeq \CC(D)^{[n]}$ by Lemma~\ref{lemA}.
By assumption, $A:=\OOO(Y)_{D}$ is a discrete valuation ring with quotient field $K:=\CC(Y)$ and residue field $k:=\CC(D)$. Moreover, $R:=\OOO(Y)_{D}\otimes_{\OOO(Y)}\OOO(X)$ is a domain, finitely generated and flat over $A$, such that 
$K\otimes_{A} R = \CC(Y)\otimes_{\OOO(Y)}\OOO(X) \simeq K^{[n]}$ by Lemma~\ref{lemA}, and  
\begin{align*}
k\otimes_{A}R &= \CC(D)\otimes_{\OOO(Y)_{D}} R=\CC(D)\otimes_{\OOO(Y)}\OOO(X)
= \CC(D)\otimes_{\OOO(D)}\OOO(D)\otimes_{\OOO(Y)} \OOO(X)\\ &=\CC(D)\otimes_{\OOO(D)}\OOO(E)\simeq\CC(D)^{[n]} 
= k^{[n]}.
\end{align*}
Therefore, by Lemma~\ref{lemB}, we get $\OOO(Y)_{D}\otimes_{\OOO(Y)}\OOO(X)\simeq\OOO(Y)_{D}^{[n]}$, and the claim follows.
\end{proof}
Define $\Ybd\subset Y$ to be the union of all open subsets $U\subset Y$  such that $\phi^{-1}(U)\to U$ is a trivial bundle. 

\begin{proof}[Proof of Theorem~\ref{fibration.thm}(b)] 
It follows from Lemma~\ref{lemB}  that for a normal variety $Y$ the complement $Y\setminus\Ybd$ has codimension at least $2$. Hence, if $Y$ is a normal curve, then $\Ybd=Y$.
\end{proof}

\begin{rem}
We have shown more generally that for an affine $\Atwo$-fibration $\phi\colon X \to Y$ where $Y$ is normal the open set $\Ybd\subset Y$ where $\phi$ is a bundle has a complement of codimension at least 2.
\end{rem}

\begin{lem}\lab{lemC}
Let $\phi\colon X \to Y$ be an $\Aone$-bundle. Assume that there are two section $\sigma,\tau\colon Y \to X$ such that $\sigma(y)\neq\tau(y)$ for all $y\in Y$. Then the bundle is trivial.
\end{lem}
\begin{proof} Given two points $a,b\in\Aone$, $a\neq b$, there is a uniquely defined morphism $\rho_{a,b}\colon \Aone \to \CC$ such that $\rho_{a,b}(a)=0$ and $\rho_{a,b}(b)=1$, and this morphism is an isomorphism. 
Now define a map $\eta\colon X \to \CC$ by
$$
\eta(x):= \rho_{\sigma(y),\tau(y)}(x) \quad\text{where }y:=\phi(x).
$$
This map is well-defined and induces an isomorphism on every fiber $\phi^{-1}(y)$. We claim that $\eta\colon X \to \CC$ is a morphism. This is obvious if  the bundle is trivial, hence follows in general, because the bundle is locally trivial. Now we claim that the morphism $(\eta,\phi)\colon X \to \CC\times Y$ is an isomorphism. Again, this is obvious if the bundle is trivial, and thus follows in general from the local triviality.
\end{proof}

\begin{proof}[Proof of Theorem~\ref{fibration.thm}(a)] 
Define $\Yt:=X \times_{Y} X$ and let $\psi\colon \Yt \to Y$ be the canonical morphism $\psi(x,x'):=\phi(x)$ ($=\phi(x')$). By definition, $\psi$ is smooth and the pull-back fibration $\phit\colon \Xt:=\Yt\times_{Y}X\to \Yt$ has two sections $\sigma,\tau\colon \Yt\to\Xt$, $\sigma(x,x'):=(x,x',x)$ and $\tau(x,x'):=(x,x',x')$. These sections are disjoint on $\Yt':=\Yt\setminus\{(x,x)\mid x\in X\}$ where $\psi'\colon\Yt'\to Y$ is still smooth and surjective. Now it suffices to prove that over any affine open set $U\subset \Yt'$ the fibration $\phit^{-1}(U)\to U$ is a trivial bundle. 

Lemma~\ref{lemC} implies that $\phi^{-1}(\Ubd)\simeq \Ubd\times \Aone$. Since the complement $Y\setminus\Ybd$ has codimension at least $2$ the same is true for $U\setminus\Ubd$ and for $\phit^{-1}(U)\setminus \phit^{-1}(\Ubd)$. But $U$ and $\phit^{-1}(U)$ are normal affine varieties, and so finally we get
$$
\OOO(\phit^{-1}(U)) = \OOO(\phit^{-1}(\Ubd))\simeq\OOO(\Ubd\times\Aone)=\OOO(\Ubd)[t] = \OOO(U)[t],
$$
hence $\phit^{-1}(U) \simeq U \times \Aone$.
\end{proof}

\par\bigskip\bigskip

\begin{thebibliography}{KKMLR97}

\bibitem[BCW82]{BaCoWr1982The-Jacobian-conje}
Hyman Bass, Edwin~H. Connell, and David Wright, \emph{The {J}acobian
  conjecture: reduction of degree and formal expansion of the inverse}, Bull.
  Amer. Math. Soc. (N.S.) \textbf{7} (1982), no.~2, 287--330. \MR{663785
  (83k:14028)}

\bibitem[BCW77]{BaCoWr1976Locally-polynomial2}
\bysame, \emph{Locally polynomial algebras are
  symmetric algebras}, Invent. Math. \textbf{38} (1976/77), no.~3, 279--299.
  \MR{0432626 (55 \#5613)}

\bibitem[FM10]{FuMa2010A-characterization}
Jean-Philippe Furter and Stefan Maubach, \emph{A characterization of semisimple
  plane polynomial automorphisms}, J. Pure Appl. Algebra \textbf{214} (2010),
  no.~5, 574--583. \MR{2577663}

\bibitem[Har77]{Ha1977Algebraic-geometry}
Robin Hartshorne, \emph{Algebraic geometry}, Graduate Texts in Mathematics, No.
  52, Springer-Verlag, New York, 1977.

\bibitem[Jan87]{Ja1987Representations-of}
Jens~Carsten Jantzen, \emph{Representations of algebraic groups}, Pure and
  Applied Mathematics, vol. 131, Academic Press Inc., Boston, MA, 1987.
  \MR{899071 (89c:20001)}

\bibitem[Jan03]{Ja2003Representations-of}
\bysame, \emph{Representations of algebraic groups}, second ed., Mathematical
  Surveys and Monographs, vol. 107, American Mathematical Society, Providence,
  RI, 2003. \MR{2015057 (2004h:20061)}

\bibitem[Kam75]{Ka1975On-the-absence-of-}
T.~Kambayashi, \emph{On the absence of nontrivial separable forms of the affine
  plane}, J. Algebra \textbf{35} (1975), 449--456. \MR{0369380 (51 \#5613)}

\bibitem[Kam79]{Ka1979Automorphism-group}
\bysame, \emph{Automorphism group of a polynomial ring and algebraic group
  action on an affine space}, J. Algebra \textbf{60} (1979), no.~2, 439--451.
  \MR{549939 (81e:14026)}

\bibitem[KW85]{KaWr1985Flat-families-of-a}
T.~Kambayashi and David Wright, \emph{Flat families of affine lines are
  affine-line bundles}, Illinois J. Math. \textbf{29} (1985), no.~4, 672--681.
  \MR{806473 (87c:14066)}

\bibitem[KKMLR97]{KaKoMa1997C-actions-on-C3-ar}
S.~Kaliman, M.~Koras, L.~Makar-Limanov, and P.~Russell,
  \emph{$\mathbb{C}^*$-actions on $\mathbb{C}^3$ are linearizable}, Electron.
  Res. Announc. Amer. Math. Soc. \textbf{3} (1997), 63--71 (electronic).
  \MR{MR1464577 (98i:14046)}

\bibitem[KZ04]{KaZa2004Venereau-polynomia}
\bysame, \emph{V{\'e}n{\'e}reau polynomials and related fiber bundles}, J. Pure
  Appl. Algebra \textbf{192} (2004), no.~1-3, 275--286. \MR{MR2067200
  (2005d:14093)}

\bibitem[KZ01]{KaZa2001Families-of-affine}
Shulim Kaliman and Mikhail Zaidenberg, \emph{Families of affine planes: the
  existence of a cylinder}, Michigan Math. J. \textbf{49} (2001), no.~2,
  353--367. \MR{1852308 (2002e:14106)}

\bibitem[Kno91]{Kn1991Nichtlinearisierba}
Friedrich Knop, \emph{Nichtlinearisierbare {O}perationen halbeinfacher
  {G}ruppen auf affinen {R}{\"a}umen}, Invent. Math. \textbf{105} (1991),
  no.~1, 217--220. \MR{MR1109627 (92c:14046)}

\bibitem[Kra96]{Kr1996Challenging-proble}
Hanspeter Kraft, \emph{Challenging problems on affine {$n$}-space},
  Ast{\'e}risque (1996), no.~237, Exp.\ No.\ 802, 5, 295--317, S{{\'e}}minaire
  Bourbaki, Vol. 1994/95. \MR{MR1423629 (97m:14042)}

\bibitem[KK96]{KrKu1996Equivariant-affine}
Hanspeter Kraft and Frank Kutzschebauch, \emph{Equivariant affine line bundles
  and linearization}, Math. Res. Lett. \textbf{3} (1996), no.~5, 619--627.
  \MR{MR1418576 (97h:14065)}

\bibitem[KS92]{KrSc1992Reductive-group-ac}
Hanspeter Kraft and Gerald~W. Schwarz, \emph{Reductive group actions with
  one-dimensional quotient}, Inst. Hautes {\'E}tudes Sci. Publ. Math. (1992),
  no.~76, 1--97. \MR{MR1215592 (94e:14065)}

\bibitem[KP85]{KrPo1985Semisimple-group-a}
Hanspeter Kraft and Vladimir~L. Popov, \emph{Semisimple group actions on the
  three-dimensional affine space are linear}, Comment. Math. Helv. \textbf{60}
  (1985), no.~3, 466--479. \MR{MR814152 (87a:14039)}

\bibitem[Kum02]{Ku2002Kac-Moody-groups-t}
Shrawan Kumar, \emph{Kac-{M}oody groups, their flag varieties and
  representation theory}, Progress in Mathematics, vol. 204, Birkh{\"a}user
  Boston Inc., Boston, MA, 2002. \MR{1923198 (2003k:22022)}

\bibitem[Sat83]{Sa1983Polynomial-ring-in}
A.~Sathaye, \emph{Polynomial ring in two variables over a {DVR}: a criterion},
  Invent. Math. \textbf{74} (1983), no.~1, 159--168. \MR{722731 (85j:14098)}

\bibitem[Sch89]{Sc1989Exotic-algebraic-g}
Gerald~W. Schwarz, \emph{Exotic algebraic group actions}, C. R. Acad. Sci.
  Paris S{\'e}r. I Math. \textbf{309} (1989), no.~2, 89--94. \MR{MR1004947
  (91b:14066)}

\bibitem[Sha66]{Sh1966On-some-infinite-d}
I.~R. Shafarevich, \emph{On some infinite-dimensional groups}, Rend. Mat. e
  Appl. (5) \textbf{25} (1966), no.~1-2, 208--212. \MR{0485898 (58 \#5697)}

\bibitem[Sha81]{Sh1981On-some-infinite-d}
\bysame, \emph{On some infinite-dimensional groups. {II}}, Izv. Akad. Nauk SSSR
  Ser. Mat. \textbf{45} (1981), no.~1, 214--226, 240. \MR{607583 (84a:14021)}

\bibitem[Sha95]{Sh1995Letter-to-the-edit}
\bysame, \emph{Letter to the editors: ``{O}n some infinite-dimensional groups.
  {II}'' [{I}zv.\ {A}kad.\ {N}auk {SSSR} {S}er.\ {M}at.\ {\bf 45} (1981), no.\
  1, 214--226, 240; {MR}0607583 (84a:14021)]}, Izv. Ross. Akad. Nauk Ser. Mat.
  \textbf{59} (1995), no.~3, 224. \MR{1347084 (96e:14054)}

\bibitem[VD74]{VeDo1974Unipotent-group-sc}
B.~Ju. Ve{\u\i}sfe{\u\i}ler and I.~V. Dolga{\v{c}}ev, \emph{Unipotent group
  schemes over integral rings}, Izv. Akad. Nauk SSSR Ser. Mat. \textbf{38}
  (1974), 757--799. \MR{0376697 (51 \#12872)}

\bibitem[vdK53]{Ku1953On-polynomial-ring}
W.~van~der Kulk, \emph{On polynomial rings in two variables}, Nieuw Arch.
  Wiskunde (3) \textbf{1} (1953), 33--41. \MR{0054574 (14,941f)}
  
\end{thebibliography}
%

\end{document}